\documentclass{amsart}
\usepackage[utf8]{inputenc}
\usepackage{amsmath,amssymb,fullpage}
\usepackage{halloweenmath}
\usepackage{amsthm}
\usepackage{amsfonts,newlfont,url,xspace}
\usepackage{stmaryrd}
\usepackage[dvipsnames]{xcolor}
\usepackage{graphics,graphicx,verbatim}
\usepackage{float}
\usepackage{enumitem}
\usepackage{hyperref, cleveref}
\usepackage{soul}
\usepackage{tikz}
\usepackage{subcaption}
\usetikzlibrary{arrows.meta} 
\newcommand{\seqnum}[1]{\href{http://oeis.org/#1}{\underline{#1}}}

\usepackage{ifxetex,ifluatex}
\if\ifxetex T\else\ifluatex T\else F\fi\fi T%
  \usepackage{fontspec}
\else
  \usepackage[T1]{fontenc}
  
  \usepackage{lmodern}
\fi

\usepackage{hyperref}

\usepackage{cancel}

\usepackage[colorinlistoftodos]{todonotes}

\newcommand{\Av}{\mathsf{Av}} 
\newcommand{\IT}{\mathrm{LPF}} 

\newcommand{\BSP}{\mathrm{BSP}}
\newcommand{\GBSP}{g\mathrm{BSP}}

\newcommand{\mB}{\mathcal B}

\newcommand{\minB}{\mathrm{min}_\mB}
\newcommand{\maxB}{\mathrm{max}_\mB}

\newcommand{\arms}{\mathrm{Arms}}
\newcommand{\legs}{\mathrm{Legs}}

\newcommand{\PF}{\mathrm{PF}}

\newcommand{\Sym}{\mathfrak{S}}

\newcommand{\Out}{\mathcal{O}}

\newcommand{\out}[1]{\mathcal{O}}

\usepackage{amsrefs}

\theoremstyle{definition} 
\newtheorem{theorem}{Theorem}[section]
\newtheorem{corollary}{Corollary}[section]
\newtheorem{proposition}{Proposition}[section]

\newtheorem{example}{Example}

\newtheorem{lemma}[theorem]{Lemma}

\theoremstyle{definition}
\newtheorem{definition}{Definition}[section]

\theoremstyle{remark}
\newtheorem*{remark}{Remark}


\author[Beerbower]{Melissa Beerbower}
\address[M.~Beerbower]{{Department of Mathematical Sciences, University of Wisconsin-Milwaukee, Milwaukee, WI 53211}}
\email{\textcolor{blue}{\href{mailto: beerbow2@uwm.edu}{beerbow2@uwm.edu}}}

\author[Elder]{Jennifer Elder}
\address[J.~Elder]{Department of Computer Science, Mathematics and Physics, Missouri Western State University, St. Joseph, MO 64507}
\email{\textcolor{blue}{\href{mailto:jelder8@missouriwestern.edu}{jelder8@missouriwestern.edu}}}
 
\author[Harris]{Pamela E. Harris}
\address[P.~E.~Harris]{Department of Mathematical Sciences, University of Wisconsin-Milwaukee, Milwaukee, WI 53211}
\email{\textcolor{blue}{\href{mailto:peharris@uwm.edu}{peharris@uwm.edu}}}
\thanks{This work was supported in part by a Simons Travel Support for Mathematicians award \#00007621.}

\author[Lavene]{Ilana Lavene}
\address[I.~Lavene]{Department of Mathematical Sciences, University of Delaware, Newark, DE 19711}
\email{\textcolor{blue}{\href{mailto:ilanalavene@gmail.com}{ilanalavene@gmail.com}}}

\author[Martinez]{Lucy Martinez}
\address[L.~Martinez]{Department of Mathematics, Rutgers University, Piscataway, NJ 08854}
\email{\textcolor{blue}{\href{mailto:lucy.martinez@rutgers.edu}{lucy.martinez@rutgers.edu}}}
\thanks{L.~Martinez was supported by the NSF Graduate Research Fellowship Program under Grant No. 2233066.}

\author[Martinson]{Adam Martinson}
\address[A.~Martinson]{Department of Mathematics, University of Maryland, College Park, MD 20742}
\email{\textcolor{blue}{\href{mailto:adammartinson123@gmail.com}{adammartinson123@gmail.com}}}

\author[Oldham]{Molly Oldham}
\address[M.~Oldham]{Department of Mathematical Sciences, University of Delaware, Newark, DE 19711}
\email{\textcolor{blue}{\href{mailto:molly.oldham343@gmail.com}{molly.oldham343@gmail.com}}}

\begin{document}
\title{Lehmer Parking Functions and Their Outcomes}

\begin{abstract}
We introduce \emph{Lehmer parking functions} and study their set of parking outcomes. Our main results establish that the number of outcomes of Lehmer parking functions of length $n$ is given by a Bell number, which is exactly the number of set partitions of an $n$ element set. We also show that the number of outcomes of weakly decreasing Lehmer parking functions is given by a Catalan number, which corresponds to a subset of set partitions on a set with $n$ elements referred to as non-intersecting set partitions.
\end{abstract}

\maketitle
\section{Introduction}
The study of patterns and discrete statsitics begins with permutations.
We write permutations in one-line notation $\sigma=\sigma_1\sigma_2\cdots\sigma_n$, and we let $\mathfrak{S}_n$ denote the set of all permutations of $[n]=\{1,2,\ldots,n\}$.
Pattern containment and avoidance is a very active field of study, as is the study of discrete statistics on permutations. For recent work in these areas, we recommend the conference proceedings of Permutation Patterns: the international conference on these subjects \cite{Perm_patterns_conf}. 

In our work, we begin by recalling that an \textit{inversion} of a permutation $\sigma=\sigma_1\sigma_2\cdots\sigma_n$ is an ordered pair of indices $(i,j)$ such that $i<j$ and $\sigma_i>\sigma_j$. 
The number of inversions of a permutation is called the \textit{inversion number} and measures how out of order the permutation is from the identity permutation $12\cdots n$. 
For example, $12345$ has zero inversions because it is completely in ascending order, while $13245$ has one inversion because $2$ and $3$ have switched and are in descending order.
Muir \cite{tmuir} credits Cramer's 1750 work \cite{gcramer} on solving linear equations via determinants, as the first to use the concept of inversions.

Given a permutation, a related concept is its \textit{inversion table}, also known as its \textit{Lehmer code}.
Inversion tables have a rich history in mathematics across a variety of fields. They have been studied in the context of tree permutations \cite{10.1145/100348.100370} and in a simple proof of MacMahon's classical result of the equidistribution of the major index and inversion statistics \cite{mjcollins}. 
Inversion tables are characterized and studied as subexcedent sequences \cite{MR3207475}. They have also been used in studying minimal left coset representatives for Weyl groups of classical type \cite{MR1834087} and in affine permutations of type $A$ \cite{MR1392503}.

Recall that the inversion table of a permutation $\sigma=\sigma_1\sigma_2\cdots\sigma_{n}\in\mathfrak{S}_n$ is defined as the tuple $\alpha=(a_1,a_2,\ldots,a_{n})$, where $a_i$ is the number of entries $j$ in $\sigma$ to the left of $i$ satisfying $j>i$ \cite[p. 30]{StanleyECVol1}.
For example, the permutation $\sigma=524613\in\mathfrak{S}_6$ has inversion table $\beta=(4,1,3,1,0,0)$.
It turns out that inversion tables of permutations of $[n]$ are exactly the $n!$ tuples of length $n$ satisfying $0\le a_i\le n-i$ for all $i\in[n]$;
moreover, the function $\sigma\mapsto\alpha$ from permutations of $[n]$ to their inversion tables is a canonical bijection \cite[Proposition 1.3.12]{StanleyECVol1}.

In this paper, we study $[0,n-1]$-valued inversion tables by adding $1$ to each entry and then consider those tuples as $[n]$-valued parking functions. 
Before giving our formal definition, we recall that a tuple $\alpha=(a_1,a_2,\ldots,a_{n})\in[n]^n$ is a \textit{parking function} if and only if the weakly increasing rearrangement $\alpha^\uparrow=(a_1',a_2',\ldots,a_{n}')$ of $\alpha$ satisfies $a_i'\leq i$ for all $i\in[n]$. We refer to this as the \textit{inequality characterization} of parking functions.
For example, $(2,2,1)$ is a parking function and $(2,2,3)$ is not.
We now introduce our main object of study.
\begin{definition}\label{def:Lehmer parking function}
    A \textit{Lehmer parking function} is a tuple $\alpha=(a_1,a_2,\ldots,a_{n})\in[n]^n$, where $a_i\le n-i+1$ for all $i\in[n]$. 
    We let $\IT_n$ denote the set of Lehmer parking functions of length $n$.
\end{definition}
For example, the inversion table $\beta=(4,1,3,1,0,0)$, after adding one to all the entries, yields the Lehmer parking function $\alpha=(5,2,4,2,1,1)$. 
Our first result establishes that all Lehmer parking functions are, as their name indicates, parking functions.
\begin{lemma}\label{lem:pf}
    For $n\geq 1$ we have $\IT_n\subseteq\PF_n$.
\end{lemma}
\begin{proof}
By definition, $\alpha=(a_1,a_2,\ldots,a_n)\in\IT_n$ satisfies $a_i\leq n-i+1$ for all $i\in[n]$. To show that $\alpha$ is a parking function, we show that the weakly increasing rearrangement $\alpha^\uparrow=(a_1',a_2',\ldots,a_n')$ of $\alpha$ satisfies $a_i'\leq i$ for all $i\in [n]$. To do this, we simply show that the reverse of $\alpha$ satisfies this inequality condition; this implies the result. If $\beta$ is the reverse of $\alpha$, then $b_i=a_{n-i-1}\le n-(n-i-1)-1=i$, as desired.
\end{proof}

Moreover, just as with permutations and their inversion tables, we can visualize a Lehmer parking function $\alpha=(a_1,a_2,\ldots,a_n)$ on an $n\times n$ grid diagram, where we place vertices at the points $(a_i,i)$ with $i\in[n]$. 
We illustrate the example $\alpha=(5,2,4,2,1,1)$ in \Cref{fig:Lehmer parking function}. 
Unlike permutations, we let $i$ take the vertical axis and $a_i$ the horizontal axis to accommodate their interpretation using the tuples as parking preferences and using a parking procedure (see \Cref{sec:background on LPFs} for more details). 
For an example of the parking outcome (a permutation encoding the final order of the cars parking on the street; see \Cref{sec:background on LPFs}), see the illustration in \Cref{fig:parking a LPF} for the Lehmer parking function $\alpha=(5,2,4,2,1,1)$.

\begin{figure}[H]
    \centering
    \includegraphics[width=0.3\linewidth]{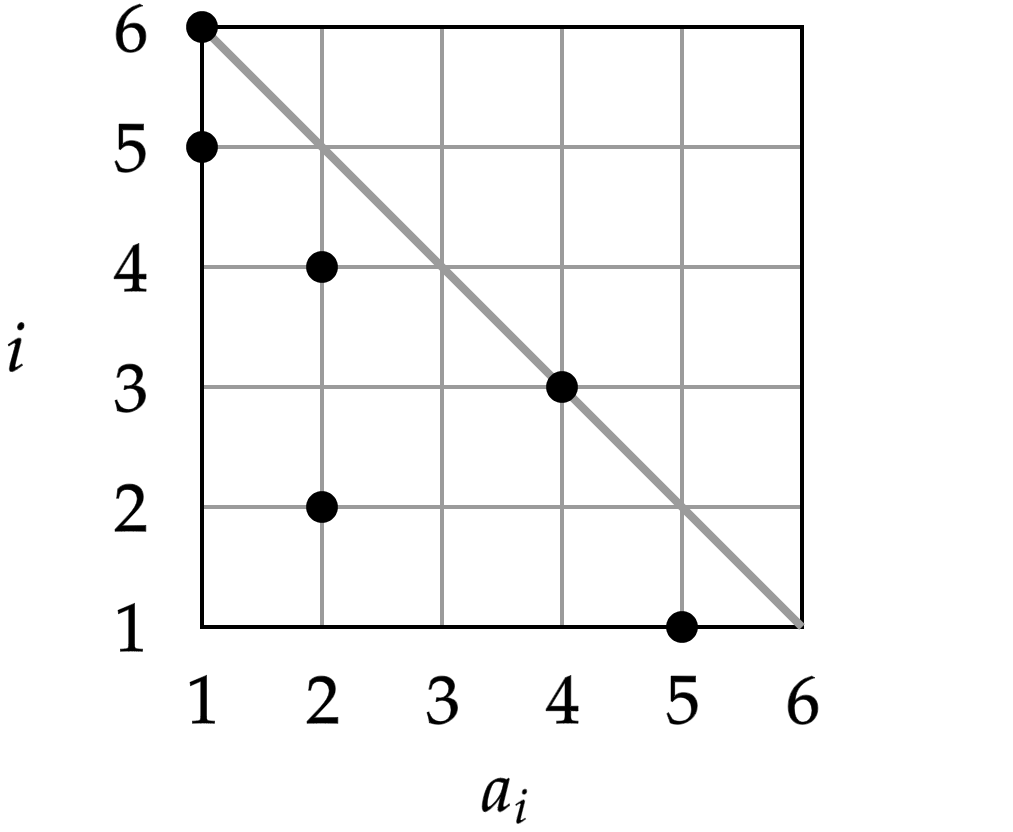}
    \caption{The grid diagram for the Lehmer parking function $\alpha=(5,2,4,2,1,1)$. The entry in each row lies at or below the antidiagonal given by $a_i=n-i+1$.
    }
    \label{fig:Lehmer parking function}
\end{figure}

\begin{figure}[h]
         \centering 
         \resizebox{\textwidth}{!}{
         \begin{tikzpicture}
         \node at(0,0){\includegraphics[width=1in]{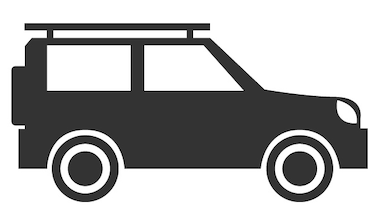}}; 
         \node at(3,0)
         {\includegraphics[width=1in]{car.png}}; \node at(6,0)
         {\includegraphics[width=1in]{car.png}}; \node at(9,0)
         {\includegraphics[width=1in]{car.png}}; \node at(12,0)
          {\includegraphics[width=1in]{car.png}};
          \node at(15,0)
          {\includegraphics[width=1in]{car.png}};
         \node at(0,-.15){\textcolor{white}{\textbf{5}}}; 
         \node at(3,-.15){\textcolor{white}{\textbf{2}}}; 
         \node at(6,-.15){\textcolor{white}{\textbf{4}}}; 
         \node at(9,-.15){\textcolor{white}{\textbf{3}}}; 
         \node at(12,-.15){\textcolor{white}{\textbf{1}}}; 
         \node at(15,-.15){\textcolor{white}{\textbf{6}}}; 
         \node at(.9,1.)
         {\includegraphics[width=.5in]{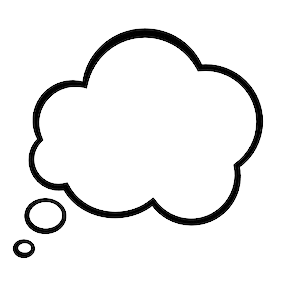}}; 
         \node at(3.9,1.)
         {\includegraphics[width=.5in]{callout.png}}; 
         \node at(6.9,1.)
         {\includegraphics[width=.5in]{callout.png}}; 
         \node at(9.9,1.)
         {\includegraphics[width=.5in]{callout.png}}; 
         \node at(12.9,1.)
         {\includegraphics[width=.5in]{callout.png}}; 
         \node at(15.9,1.)
         {\includegraphics[width=.5in]{callout.png}}; 
         \node at(.9,1.1){\textbf{1}}; 
         \node at(3.9,1.1){\textbf{2}}; 
         \node at(6.9,1.1){\textbf{2}}; 
         \node at(9.9,1.1){\textbf{4}}; 
         \node at(12.9,1.1){\textbf{5}}; 
         \node at(15.9,1.1){\textbf{1}}; 
         \draw[ultra thick](-1.25,-.6)--(1.25,-.6); 
         \draw[ultra thick](1.75,-.6)--(4.25,-.6); 
         \draw[ultra thick](4.75,-.6)--(7.25,-.6); 
         \draw[ultra thick](7.75,-.6)--(10.25,-.6); 
         \draw[ultra thick](10.75,-.6)--(13.25,-.6); 
         \draw[ultra thick](13.75,-.6)--(16.25,-.6); 
         \node at(0,-1){\textbf{1}}; 
         \node at(3,-1){\textbf{2}}; 
         \node at(6,-1){\textbf{3}}; 
         \node at(9,-1){\textbf{4}}; 
         \node at(12,-1){\textbf{5}}; 
         \node at(15,-1){\textbf{6}}; 
         \end{tikzpicture} 
         }
    \caption{Parking cars using the Lehmer parking function $\alpha=(5,2,4,2,1,1)$ results in the cars parking in order (i.e., in outcome) $524316\in\mathfrak{S}_6$. Car image designed by Freepik and callout designed by macrovector / Freepik.}
    \label{fig:parking a LPF}
\end{figure}

\noindent By \Cref{def:Lehmer parking function}, Lehmer parking functions are exactly the tuples whose diagrams have one entry in each row, with all entries lying at or below the \textit{antidiagonal} from $(1,n)$ to $(n,1)$.
In contrast, the grid diagram of a permutation requires that there exists a unique entry appearing in each row and column.

\begin{remark}
    Initially, we used the term inversion tables, as we were focused on permutation literature which used grid diagrams such as \Cref{fig:Lehmer parking function} to study inversions. However, we adopt the name Lehmer parking functions rather than inversion table parking functions so as not to confuse our study with that of inversions \textit{in} parking functions. For more information on work related to inversions in parking functions, see \cite{celano2025inversionsparkingfunctions, celano2025statisticsellintervalparkingfunctions}.
\end{remark}

The paper is organized as follows: In \Cref{sec:background on LPFs}, we fully characterize all possible outcomes of a Lehmer parking function as a class of pattern avoiding permutations (\Cref{thm: characterization of inversion table outcomes}). Once we have this classification, in \Cref{sec:bijection} we prove that the Lehmer parking functions of length $n$ are enumerated by the $n$th Bell number (\cite[\seqnum{A000110}]{OEIS}), through a bijection with the standard set partitions of $[n]$ and a new set of objects called \emph{$g$-balanced spaced parenthesizations} (\Cref{thm:main bijection} and \Cref{def:bsp_n}). In \Cref{sec:catalan}, we provide enumerations through the Catalan numbers (\cite[\seqnum{A000108}]{OEIS}) for the weakly increasing outcomes of Lehmer parking functions of length $n$ (\Cref{thm:weakly_dec_lehmer})

\section{Characterization of outcomes of Lehmer parking functions}\label{sec:background on LPFs} 

A parking function $a=(a_1,a_2,\ldots,a_n)\in[n]^n$ can also be interpreted as a tuple indexing \textit{preferences} of cars labeled $1$ through $n$ in line to park on a one-way street with spots labeled $1$ through $n$. Starting with car $1$, each car $i\in[n]$ parks in the first empty spot at or after its preference $a_i$. Parking functions are exactly the tuples for which all cars park without driving past the last spot. Following \Cref{fig:parking functions}, $(2,2,1)$ is a parking function because all cars can park. Meanwhile, $(2,2,3)$ is not a parking function since spot $1$ is undesired by every car and car three fails to park.
It follows from this interpretation that a tuple $\alpha\in[n]^n$ is a parking function if and only if it contains at least $i$ values less than or equal to $i$ for every $i\in[n]$; this is equivalent to the inequality characterization of parking functions.

The \textit{outcome} $\Out(\alpha)$ of a parking function $\alpha$ is the permutation $\pi=\pi_1\pi_2\cdots\pi_n$, where $\pi_i=j$ denotes that the spot $i$ is occupied by car $j$ after the parking process. For example, $\Out((2,2,1))=312$ (see \Cref{fig:parking functions}) and $\Out((3,1,1))=231$. 
As another example, we illustrate the parking procedure using the Lehmer parking function $\alpha=(5,2,4,2,1,1)$ in \Cref{fig:parking a LPF}, and we note that the parking \textit{outcome} is the permutation consisting of the cars' labels, in this case the permutation $524316$.

\begin{figure}[H]
    \centering
    \includegraphics[width=0.8\linewidth]{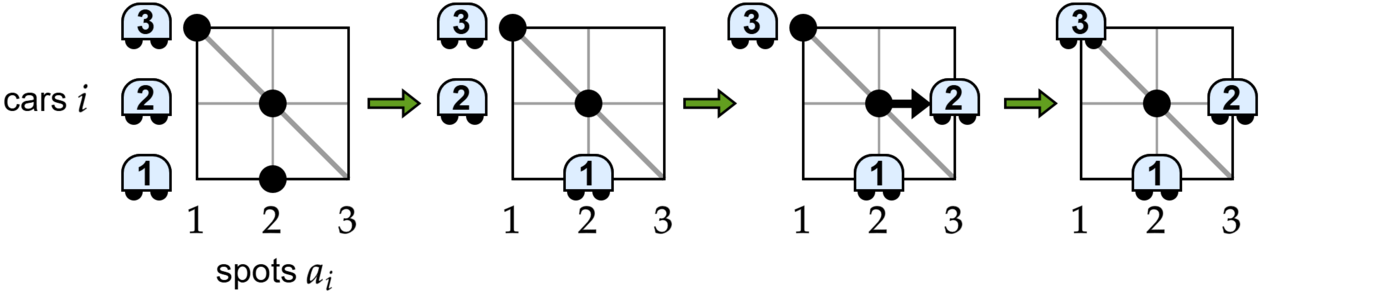}
    \caption{We park the preference list $\alpha=(2,2,1)$. Here, the dot $a_1=2$ on the diagram conveys that the preference of car $1$ is spot $2$. In the third diagram, car $2$ is forced to park in spot $3$ because spot $2$ is occupied by car $1$. The outcome of $\alpha$ is the permutation $\pi=312$.}
    \label{fig:parking functions}
\end{figure}

Every permutation $\pi$ can be realized as the outcome of a parking function (in particular, its inverse $\pi^{-1}$). 
That is, the set of all outcomes of parking functions of length $n$, $\Out(\PF_n)$, is equal to $\Sym_n$. However, this fact does not hold for outcomes of \textit{Lehmer} parking functions in general.
For example, assume for contradiction that $\pi=132$ is the outcome of some Lehmer parking function $\alpha=(a_1,a_2,a_3)$. In \Cref{fig:ipfo non example} we illustrate the parking procedure according to $\pi$. First, car 1 parks in spot 1 and car 2 parks in spot 3. But by \Cref{def:Lehmer parking function}, we have $a_2\leq 2$, so the preference $a_2$ of car $2$ may only be $1$ or $2$. Because spot $1$ is occupied by car $1$ and spot $2$ is empty, either preference causes car $2$ to park in spot $2$. This contradicts that $\pi_3=2$. Hence, $\pi=132$ cannot be the outcome of a Lehmer parking function.

\begin{figure}[H]
    \centering
    \includegraphics[width=0.65\linewidth]{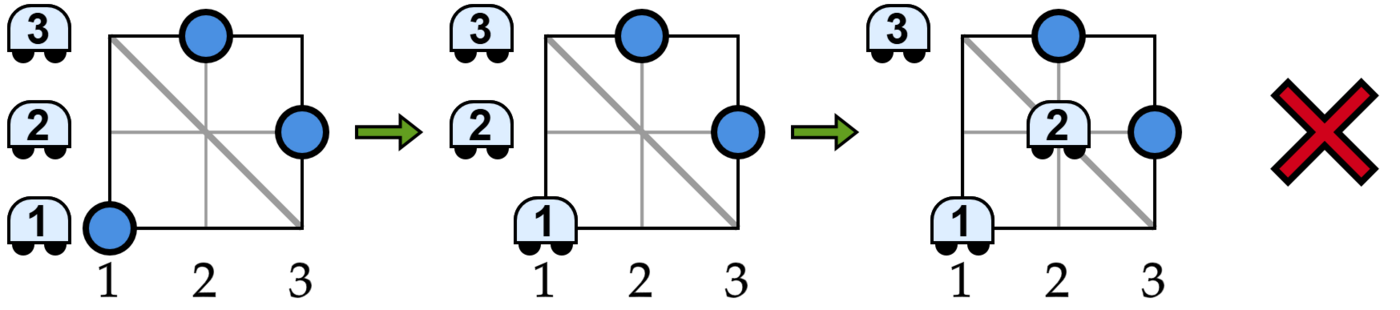}
    \caption{We display the desired outcome $\pi=132$ in blue. If car $1$ parks in spot $1$, it is impossible for car $2$ to park in spot $3$. Thus, $\pi$ is not the outcome of a Lehmer parking function.}
    \label{fig:ipfo non example}
\end{figure}

Given that not every permutation is the outcome of a Lehmer parking function, a natural question is: \textit{Which permutations are the outcomes of Lehmer parking functions?} We answer this question by giving a complete characterization of the set of outcomes of Lehmer parking functions, denoted $\Out (\IT_n)$. This characterization is best described by a property which is visualized on the \textit{permutation diagram} of a permutation $\pi$. We begin with a definition. 
\begin{definition}\label{def:arm_leg_drawing}
    As usual, one can plot a permutation $\pi=\pi_1\pi_2\cdots\pi_n$ on the grid $[n]\times [n]$ by plotting the points $(i,\pi_i)$ for each $i\in[n]$.  
    On this diagram, we additionally draw the following lines:
        \begin{itemize}
            \item The \textit{antidiagonal} $y=-x+(n+1)$, as our diagrams have the upper left corner $(1,n)$
            \item For each entry $(i,\pi_i)$ lying at or above the antidiagonal, that is, those with coordinates satisfying $\pi_i\ge-i+(n+1)=n-i+1$:
            \begin{itemize}
                \item the horizontal line segment connecting the point $(i,\pi_i)$ to the antidiagonal, which we call the \textit{arm}, and 
                \item the vertical line segment connecting the point $(i,\pi_i)$ to the antidiagonal, which we call the \textit{leg}.
            \end{itemize}
            (Note: 
            in our illustration, we extend the arms and legs slightly beyond the antidiagonal so that those coming from entries $(i,\pi_i)$ lying \textit{on} the antidiagonal are more clearly visible.)
        \end{itemize}
    The resulting diagram is called the \textit{arm-leg diagram} of $\pi$. We say that the arm-leg diagram is \textit{non-intersecting} if the arms and the legs do not intersect each other. We note that the arm-leg diagram of a permutation $\pi$ is non-intersecting if and only if it avoids the pattern $n-j+1<n-i+1\le\pi_{j}<\pi_{i}$; see \Cref{perm:running_ex_1}.
\end{definition}

    \begin{figure}[H]
        \centering
        \includegraphics[width=0.7\linewidth]{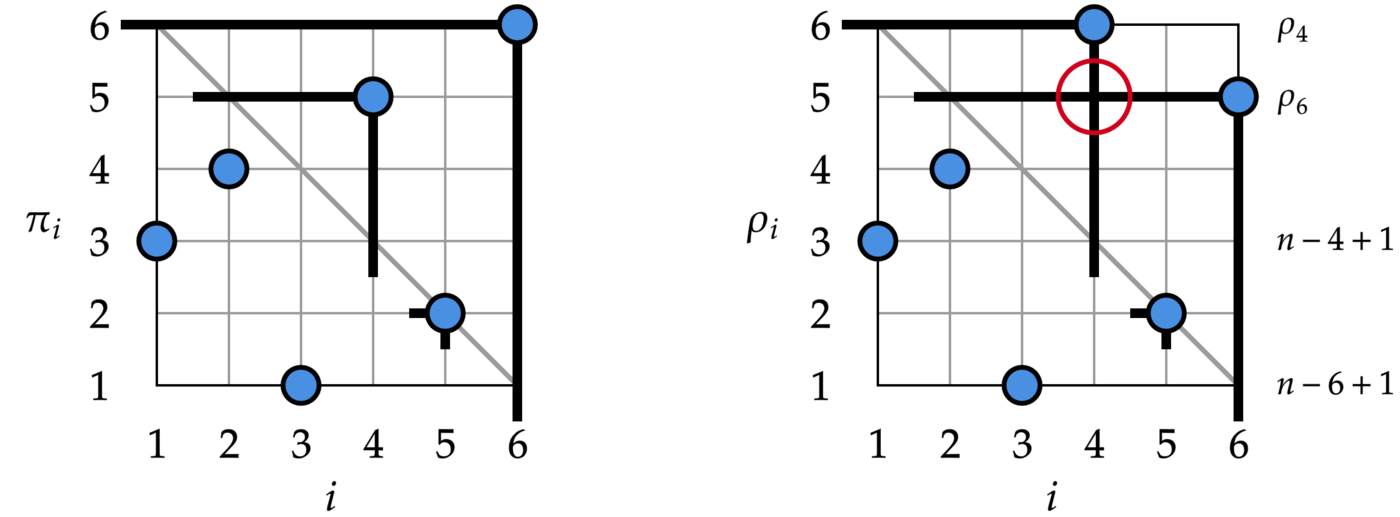}
        \caption{The arm-leg diagrams of $\pi=341526$ and $\rho=341625$, respectively. In blue we identify the entries $(i,\pi_i)$ or $(i,\rho_i)$. The right diagram is intersecting, which can be seen by its containment of the pattern $n-j+1<n-i+1\le\rho_j<\rho_i$ for $i=4$, $j=6$. }
         \label{perm:running_ex_1}
    \end{figure}

Most important to our study are permutations with non-intersecting arm-leg diagrams. For example,
    the arm-leg diagrams of $\pi=341526$ and $\rho=341625$ are given in \Cref{perm:running_ex_1}. Notice that peaks drawn from the points at or above the antidiagonal of $\pi$ do not intersect, indicating that the arm-leg diagram of $\pi$ is non-intersecting. In contrast, the arm-leg diagram of $\rho$ is intersecting as marked by the red-circled intersection in its figure.

Our first result connects the outcome of a Lehmer parking function and its arm-leg diagram.

\begin{proposition}\label{prop:forward1}
    The arm-leg diagram of any $\pi\in \Out (\IT_n)$ is non-intersecting.
\end{proposition}
\begin{proof}
    Suppose that $\pi=\pi_1\pi_2\cdots\pi_n$ is the outcome of a Lehmer parking function $\alpha=(a_1,a_2,\ldots,a_n)$. Suppose that the arm-leg diagram of $\pi$ has an arm $[n-i+1,i]\times\{\pi_i\}$ coming from the point $(i,\pi_i)$.
        By \Cref{def:Lehmer parking function}, we obtain
    \begin{align*}
        a_{\pi_i}\le~&n-\pi_i+1,
        \intertext{and since $(i,\pi_i)$ is at or above the antidiagonal, we obtain $\pi_i\ge n-i+1$, or}
         &n-\pi_i+1 \leq i.
    \end{align*}
    Because car $\pi_i$ has preference $a_{\pi_i}\le n-\pi_i+1$ and parks in spot $i\ge n-\pi_i+1$, we know that all spots $j\in[n-\pi_i+1,i-1]$ are occupied by cars from $\pi_1$ through $\pi_{i-1}$ when $\pi_i$ parks. Thus, $\pi_j<\pi_i$ for all $j\in[n-\pi_i+1,i-1]$, so there is no entry above the arm $[n-i+1,i]\times\{\pi_i\}$ with a leg that intersects it.
\end{proof}

\begin{figure}[H]
    \centering
    \includegraphics[width=0.36\linewidth]{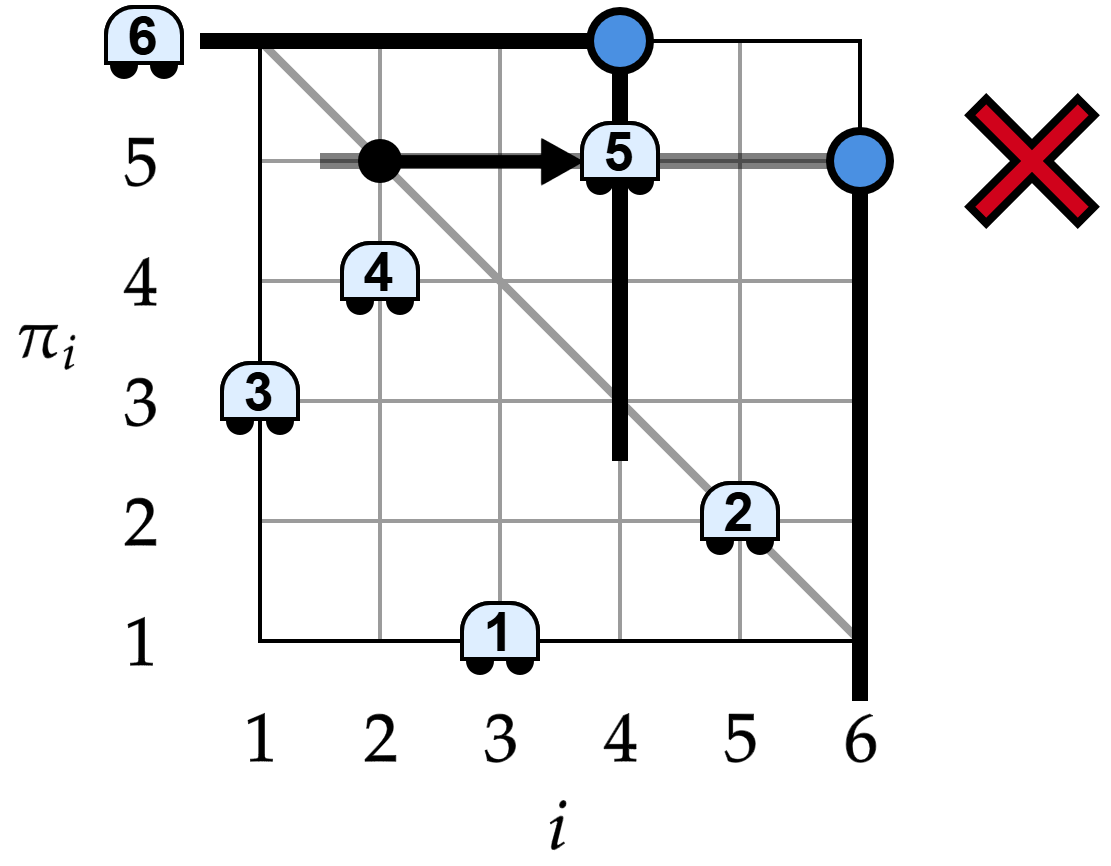}
    \caption{We display the contrapositive of \Cref{prop:forward1}: if $\pi$ has an intersecting arm-leg diagram, then it is not the outcome of a Lehmer parking function. Here, even with preference $a_5=2$, car $5$ stops at the empty spot $4$, which disagrees with $\pi_6=5$.}
\end{figure}

Next, we establish the converse of \Cref{prop:forward1}.

\begin{proposition}\label{prop:backwards1}
   If the arm-leg diagram of a permutation $\pi$ is non-intersecting, then $\pi \in \Out (\IT_n)$. 
\end{proposition}
\begin{proof}
    Suppose that the arm-leg diagram of the permutation $\pi=\pi_1\pi_2\cdots\pi_n$ is non-intersecting.
    Note that if car $k$ parks in spot $i$ then $\pi_i=k$, or alternatively, $\pi_k^{-1}=i$. Thus, $\pi$ is the outcome where each car $k$ parks in spot $\pi_k^{-1}$.
    If we only wanted a (not necessarily Lehmer) parking function $\beta=(b_1,b_2,\ldots,b_n)$ with outcome $\pi$,
    it would suffice to set $b_k=\pi_k^{-1}$ for all $k\in[n]$.
    Since we need a \textit{Lehmer} parking function $\alpha=(a_1,a_2,\ldots,a_n)$ with outcome $\pi$, we set
    $a_k=\min(\pi_k^{-1},n-k+1)$ for all $k\in[n]$, satisfying the inequality condition in \Cref{def:Lehmer parking function}.

    It remains for us to show that $\alpha$ has outcome $\pi$. We prove this inductively by showing that each car $k$ parks in spot $\pi_k^{-1}$. First, car $1$ prefers spot $a_1=\min(\pi_k^{-1},n-1+1)=\pi_k^{-1}$ and parks there. Next, assume that cars $1,2,\ldots, k-1$ park in spots $\pi^{-1}_1,\pi^{-1}_2,\ldots,\pi^{-1}_{k-1}$, respectively, leaving spots $\pi^{-1}_k,\pi^{-1}_{k+1},\ldots,\pi^{-1}_{n-1}$ empty. We consider the following cases:
    \begin{itemize}
        \item If $\pi^{-1}_k<n-k+1$, then car $k$ parks in the empty spot $a_k=\min(\pi^{-1}_k,n-k+1)=\pi^{-1}_k$.
        \item If $\pi^{-1}_k\geq n-k+1$, then $a_k=\min(\pi^{-1}_k,n-k+1)=n-k+1$. The only way car $k$ parks in $\pi_k^{-1}$ is if every spot from $n-k+1$ through $\pi^{-1}_k-1$ is occupied. 
        This is true because, by assumption that the arm-leg diagram of $\pi$ is non-intersecting, every spot $j\in[n-k+1,\pi^{-1}_k-1]$ satisfies $\pi_j<\pi_{\pi^{-1}_k}=k$.
    \end{itemize}
    This completes the inductive step and the proof.\qedhere
\end{proof}

The previous two propositions have established the following main result.

\begin{theorem}\label{thm: characterization of inversion table outcomes}
    A permutation $\pi=\pi_1\pi_2\cdots\pi_n$ is an element of $\Out(\IT_n)$ 
    if and only if its arm-leg diagram is non-intersecting, or equivalently, if and only if it
    avoids the pattern $n-j+1<n-i+1\le \pi_j<\pi_i$ for all $i<j\in[n]$.
\end{theorem}

\section{Main Bijection}\label{sec:bijection}
Having characterized the outcomes of Lehmer parking functions in the previous section, we now establish a bijection to a classical combinatorial object: set partitions. 
We recall that a \textit{set partition} of $[n]$ is a set of nonempty disjoint subsets, or \textit{blocks}, whose union is $[n]$.
\begin{theorem}\label{thm:main bijection}
    Outcomes of Lehmer parking functions of length $n$ are in bijection with set partitions of $[n]$.
\end{theorem}

Before our proof of this result, we recall that it is well-known that the number of set partitions of $[n] $ is given by 
$B_n$, the $n$-th Bell number (\cite[\seqnum{A000110}]{OEIS}).
Hence, an immediate consequence of \Cref{thm:main bijection} is as follows.
\begin{corollary}
    For all $n\geq 0$, 
    $|\Out(\IT_n)| = B_n$.
\end{corollary}

To prove \Cref{thm:main bijection}, we define a new combinatorial set (see \Cref{def:gBSP}) which parameterizes both the outcomes of Lehmer parking functions and set partitions. In particular, an element of this set is given by an arrangement of parentheses with some additional digits. In \Cref{sec:invPFoutcome}, we define this set and prove that it is in bijection with the outcomes of Lehmer parking functions. In \Cref{sec:setpartition}, we prove that this set is also in bijection with set partitions, thus proving \Cref{thm:main bijection}.

\subsection{Outcomes of Lehmer Parking Functions and Balanced Spaced Parenthesizations}\label{sec:invPFoutcome}

Our approach is as follows. In \Cref{Out(IT) to BSP}, we define a set $\BSP_n$ and a function $\Phi:\Out(\IT_n)\to\BSP_n$. In \Cref{BSP to Out(IT)} we see that this function is surjective but not injective. To fix this, we extend $\BSP_n$ and $\Phi$ to a new set $\GBSP_n$ and function $\Phi':\Out(\IT_n)\to\GBSP_n$ which is a bijection.

\subsubsection{From $\Out(\IT_n)$ to $\BSP_n$}\label{Out(IT) to BSP}

Given an outcome of a Lehmer parking function, $\pi$, we define a ``peak'' of its corresponding arm-leg diagram to be made of an entry at or above the antidiagonal together with the arm and the leg attached to it.
We see that the peak visually ``contains'' a segment of the antidiagonal (see \Cref{fig:ex parenthesis}, left). 
Because the peaks of the arm-leg diagram do not intersect (\Cref{thm: characterization of inversion table outcomes}), they can nest inside of each other, much like pairs of parentheses. 

To formalize this observation, we proceed as follows.
To each coordinate point on the antidiagonal we assign a labeled space drawn as a small line segment below the antidiagonal (see \Cref{fig:ex parenthesis}, center).
These spaces are labeled consecutively with $1$ at the upper-left corner $(1,n)$ and with $n$ at the lower-right corner $(n,1)$.
For a peak $(i,\pi_i)$, the arm corresponds to an opening parenthesis at space $n-\pi_i+1$, and the leg corresponds to a closing parenthesis at space $i$.
In \Cref{fig:ex parenthesis}, center, the peak $(4,5)$ ``contains" spaces $2$ through $4$.
We introduce this technical definition next.

\begin{figure}[H]
    \centering
    \includegraphics[width=0.9\linewidth]{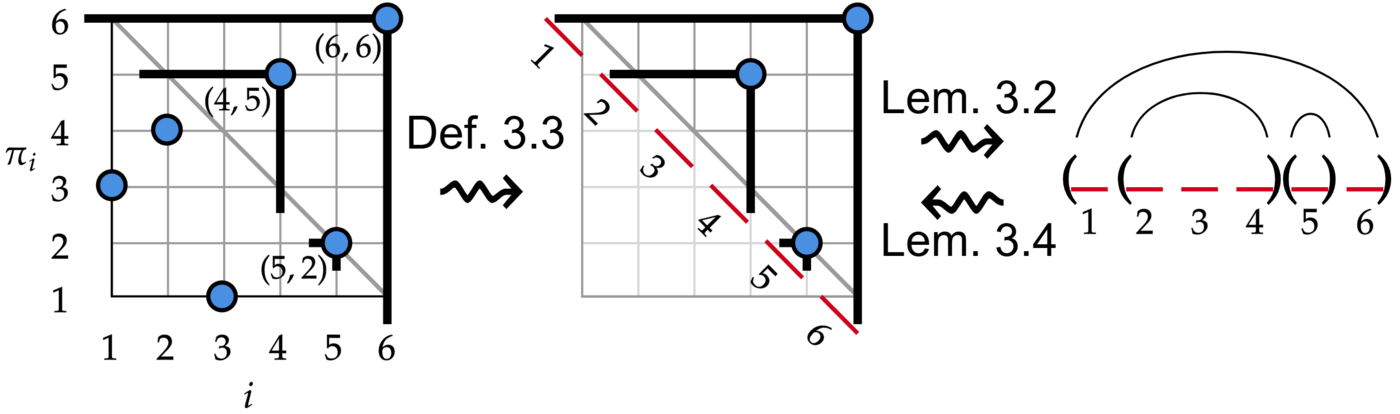}
    \caption{Peaks of the arm-leg diagram of $\pi=341526\in\Out(\IT_n)$ ``contain'' segments of the antidiagonal (left). We illustrate this by placing labeled spaces under the antidiagonal (center). Finally, the peaks of $\pi$ are assigned the pictured \textit{balanced spaced parenthesization} (right). } 
     \label{fig:ex parenthesis}
\end{figure}

\begin{definition}
    Let $\pi=\pi_1\pi_2\cdots\pi_n$ be an outcome of a Lehmer parking function of length $n$ (that is, $\pi\in  \Out(\IT_n)$).
    We recall that an entry $(i,\pi_i)$ lies at or above the antidiagonal if it satisfies $\pi_i\geq n-i+1$.
    We define
    \begin{align*}
        \arms_\pi&=\{n-\pi_i+1~:~(i,\pi_i)\mbox{ lies at or above the antidiagonal}\},\text{  and}\\
        \legs_\pi&=\{i~:~(i,\pi_i)\mbox{ lies at or above the antidiagonal}\}.
    \end{align*}
\end{definition}
\noindent For example, the arm-leg diagram of $\pi=341526$ has peaks $(4,5)$, $(5,2)$, and $(6,6)$, giving us $\arms_\pi=\{1,2,5\}$ and $\legs_\pi=\{4,5,6\}$. In \Cref{fig:ex parenthesis}, we extract the set of peaks from the arm-leg diagram and assign matching pairs of parentheses to them.

Next, we define an appropriate notion of parenthesization. 
First, we recall the conventional definition of balanced parenthesis as follows \cite[Section 1.1]{catalan_ex}. 
We start with a partial sum set equal to 0. Moving left to right, we add 1 to the partial sum after passing each opening parenthesis `(', and we subtract 1 after passing each closing parenthesis `).' 
A parenthesization is \textit{balanced} if the partial sum is always nonnegative. 
We modify this definition to capture information from outcomes of Lehmer parking functions.
\begin{definition}\label{def:bsp_n}
    We begin with $n$ spaces labeled $1$ through $n$. A \textit{spaced parenthesization} of length $n$ adds at most one opening parenthesis `(' before each space and at most one closing parenthesis `)' after each space.
    We denote this by a pair $(F,L)$ of subsets of $[n]$ of equal size; $F$ is the set of labels of the spaces which are preceded by an opening parenthesis `(', and $L$ is the set of labels of the spaces which are followed by a closing parenthesis `).' 
    We define the \textit{parenthesization depth} $d_i$ at space $i$ to be the number of `('s before the $i$-th space minus the number of `)'s before the $i$-th space. 
    Equivalently, \[d_i=|F\cap[1,i]|-|L\cap[1,i-1]|.\] 
    A spaced parenthesization $(F,L)$ is \textit{balanced} if the depth $d_i$ is positive at each space $i\in[n]$. We let $\BSP_n$ denote the set of balanced spaced parenthesizations of length $n$.
\end{definition}
We illustrate \Cref{def:bsp_n} next. 
\begin{example}\label{ex:balanced ex} 
The balanced spaced parenthesization $(\{1,3,5\},\{5,6,7\})$ can be illustrated by
\[\begin{smallmatrix}(\_&\_&(\_&\_&(\_)&\_)&\_)\\\,1&2&\,\,3&4&\,5\,&6\,\,&7\,\end{smallmatrix}.\]
For each $i\in[7]$, placing the value $d_i$ at the labeled space $i$ yields
\[\begin{smallmatrix}(\underline{\,1\,}&\underline{\,1\,}&(\underline{\,2\,}&\underline{\,2\,}&(\underline{\,3\,})&\underline{\,2\,})&\underline{\,1\,})\\\,\,1&2&\,\,3&4&\,5\,&6\,\,&7\,\,\end{smallmatrix}.\]
Because the depth $d_i$ is positive at each $i\in[7]$, all spaces are contained in some matching pair of parentheses. We can thus conclude that $1\in F$ and $7\in L$; this fact holds in general, so that $1\in F$ and $n\in L$. 
\end{example}

Examples failing the definition of balanced spaced parenthesizations include:
\begin{itemize}
    \item 
$(\{1,2\},\{4,4\}=\{4\})=(\_\ (\_\ \_\ \_))$, because it has more than one closing parenthesis after a space; 
\item $(\{1,2,4\},\{3,4\})=(\_\ (\_\ \_)\ (\_)$, because $|F|\ne|L|$; and
\item $(\{1,4\},\{2,4\})=(\_\ \_)\ \_\ (\_)$, because $d_3=0$ indicates that it is not balanced.
\end{itemize}

Next we show that the arm-leg diagram of an outcome of a Lehmer parking function corresponds to a balanced spaced parenthesization.  
To do so, first we show that the depth $d_i$ is nonnegative for each space $i\in[n]$, and then we show that $d_i>0$ for each space $i\in[n]$. 
In order to prove this result, we make some observations and prove an additional technical result. 
To begin, we rephrase \Cref{def:bsp_n} by noting that the parenthesization depth $d_i$ at a labeled space $i$ of a balanced spaced parenthesization is the number of matching pairs of parentheses containing the space $i$. 
We first show in \Cref{lem:d_i = leg intersections} that for the spaced parenthesization $(\arms_\pi,\legs_\pi)$, $d_i$ is the number of peaks of the arm-leg diagram of $\pi$ that visually ``contain'' the $i$-th space. 
For example, in \Cref{fig:ex parenthesis}, center, space $4$ is contained in the two peaks $(4,5)$ and $(6,6)$, so $d_4=2$. 
Similarly, space $5$ is contained in the two peaks $(5,2)$ and $(6,6)$, so $d_5=2$. 
However, spaces $1$ and $6$ are only contained in the peak $(6,6)$, so $d_1=d_6=1$.
In general, we note that space $i$ is visually contained in exactly the peaks which lie in the box $[i,n]\times[n-i+1,n]$ (see \Cref{fig:horizontal line for d_i} for an illustration).

\begin{lemma}\label{lem:d_i = leg intersections}
    Let $\pi\in \Out (\IT_n)$ and consider the spaced parenthesization $(\arms_\pi,\legs_\pi)$. Then for all $i\in [n]$, $d_i$ is the number of peaks of the arm-leg diagram of $\pi$ lying in the box $[i,n]\times[n-i+1,n]$. Consequently, $d_i$ is nonnegative for all $i\in[n]$.
\end{lemma}
\begin{proof}
    Fix arbitrary $i\in[n]$.
    To compute $d_i=|\arms_\pi\cap[1,i]|-|\legs_\pi\cap[1,i-1]|$, we check for each peak $(k,\pi_k)$ of the arm-leg diagram of $\pi$ whether the arm is in $[1,i]$ and whether the leg is in $[1,i-1]$. 
    By definition, the peak $(k,\pi_k)$ lies at or above the antidiagonal, so $\pi_k\ge n-k+1$, or $k\ge n-\pi_k+1$. Thus, the leg $k$ of a peak is greater than or equal to the arm $n-\pi_k+1$, so a peak can only contribute $0$ or $1$ to the total $d_i=|\arms_\pi\cap[1,i]|-|\legs_\pi\cap[1,i-1]|$.
    The peak $(k,\pi_k)$ contributes $1$ to $d_i$ if and only if the arm $n-\pi_k+1$ lies in the interval $[1,i]$ and the leg $k$ lies in the interval $[i,n]$. This happens if and only if $k\in[i,n]$ and $\pi_k\in[n-i+1,n]$, as desired.
    This establishes that $d_i$  is nonnegative at each space $i\in[n]$.
\end{proof}

\begin{figure}[H]
    \centering
    \includegraphics[width=0.35\linewidth]{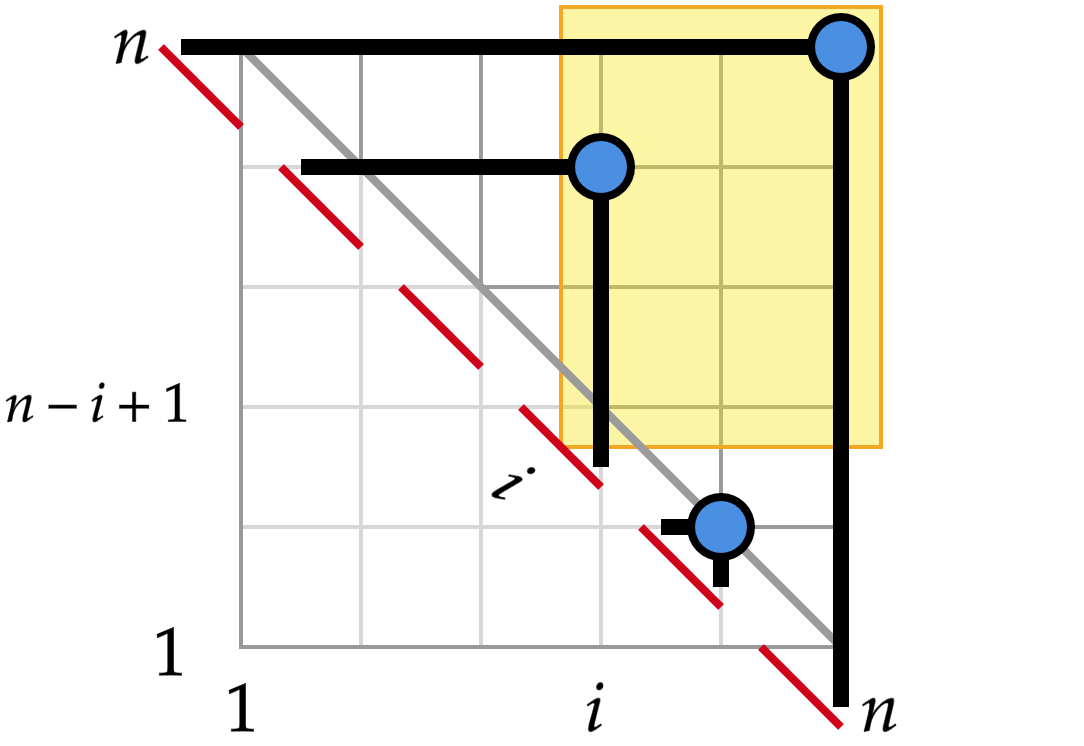}
    \caption{The $i$-th space is visually contained in two peaks; that is, the yellow box contains two entries. \Cref{lem:d_i = leg intersections} states that this is equivalent to $d_i=2$.}
    \label{fig:horizontal line for d_i}
\end{figure}

It now suffices to show that $d_i$ is \textit{positive} at each $i\in[n]$.

\begin{lemma}\label{lem:bsp connection}
    If $\pi \in \Out(\IT_n)$, then the spaced parenthesization given by $(\arms_\pi,\legs_\pi)$ is balanced, i.e., $d_i$ is positive at each space $i\in[n]$.
\end{lemma}

\begin{proof}
    Fix arbitrary $i\in[n]$. By \Cref{lem:d_i = leg intersections}, $d_i$ is positive if there is a peak in the box $[i,n]\times[n-i+1,n]$.
    In the arm-leg diagram of $\pi$, there are $i$ entries
    in the top $i$ rows (numbered $n-i+1$ to $n$), and only $i-1$ of them can lie in the leftmost $i-1$ columns (numbered $1$ to $i-1$). 
    Thus, at least one entry lies at or to the right of column $i$.
    This entry lies in the desired box $[i,n]\times[n-i+1,n]$.
\end{proof}
In the next subsection, we describe a bijection between the sets of interest $\Out (\IT_n)$ and $\GBSP_n$.

\subsubsection{From $\BSP_n$ to $\Out(\IT_n)$}\label{BSP to Out(IT)}

\Cref{lem:bsp connection} states equivalently that the function 
\begin{align}\label{eq:Phi}
    \Phi:\Out(\IT_n)\to\BSP_n,\quad\pi\mapsto(\arms_\pi,\legs_\pi)    
\end{align}
is well-defined. The next step is to show that $\Phi$ is surjective; that is, that every balanced spaced parenthesization can be written as $(\arms_\pi,\legs_\pi)$ for some $\pi$ which is an outcome of a Lehmer parking function. We accomplish this by constructing such an outcome $\pi$ in two steps. 
First, we place entries at or above the antidiagonal in \Cref{lem:balanced -> non-intersecting partial}. Second, we place entries below the antidiagonal in \Cref{lem:Phi is surjective}, completing the permutation diagram of $\pi$.

Before stating and proving \Cref{lem:balanced -> non-intersecting partial}, we build intuition using \Cref{fig:ex parenthesis}: A balanced spaced parenthesization  determines a set of matching pairs of parentheses, and each possible pair of parentheses corresponds to a particular coordinate point at or above the antidiagonal. 
For example, \Cref{fig:ex parenthesis} (right) displays the balanced spaced parenthesization $(\{1,2,5\},\{4,5,6\})$ with matching pairs $[1,6]$, $[2,4]$, and $[5,5]$ (see \Cref{fig:ex parenthesis}, center), which can be converted into peaks $(6,6)$, $(4,5)$, and $(5,2)$, respectively. 
The fact that matching pairs of parentheses nest inside of each other guarantees that the construction in \Cref{lem:balanced -> non-intersecting partial} is non-intersecting. 
We use the following definition to discuss the set of peaks in isolation:

\begin{definition}\label{def:partial}
    A \textit{partial arm-leg diagram} $\tau$ is a subset of coordinate points $(i,\tau_i)$ in the grid $[n]\times[n]$ which lie at or above the antidiagonal (i.e. $\tau_i\geq n-i+1$), with no two entries in the same row or column. We define $\arms_\tau$ and $\legs_\tau$ analogously, i.e., $\arms_\tau=\{n-\tau_i+1~:~(i,\tau_i)\in\tau\}$ and $\legs_\tau=\{i~:~(i,\tau_i)\in\tau\}$. We define $\tau$ to be non-intersecting if the arms and legs do not intersect each other, or equivalently, if it avoids the pattern $n-j+1<n-i+1\le\tau_j<\tau_i$ for $(i,\tau_i),(j,\tau_j)\in\tau$.
\end{definition}

An example of a partial arm-leg diagram is shown in the center of \Cref{fig:ex parenthesis}. We remark that for $\pi$, an outcome of a Lehmer parking function, the balanced spaced parenthesization $(\arms_\pi,\legs_\pi)$ only depends on the peaks of the arm-leg diagram of $\pi$; that is, it can written as $(\arms_\tau,\legs_\tau)$ where $\tau$ is the set of peaks of the arm-leg diagram of $\pi$.

We remark additionally that not every partial arm-leg diagram is the set of peaks of an arm-leg diagram of an outcome of a Lehmer parking function. For example, if $\tau$ is empty, then there is no entry in the top row or in the rightmost column, indicating that $\tau$ is not the set of peaks of a diagram for any $\pi\in \Out(\IT_n)$.

We are now ready for our first result.

\begin{lemma}\label{lem:balanced -> non-intersecting partial}
    If $(F,L)$ is a balanced spaced parenthesization of $n$ spaces, then there exists a unique non-intersecting partial arm-leg diagram $\tau$ of the grid $[n]\times[n]$ such that $(\arms_\tau,\legs_\tau)=(F,L)$. If $(F,L)$ has $k$ pairs of parentheses, then $\tau$ has $k$ entries. 
\end{lemma}

In order to prove \Cref{lem:balanced -> non-intersecting partial}, we must first be able to extract the matching pairs of parentheses from a balanced spaced parenthesization.
To this end, we introduce an equivalent way to describe balanced spaced parenthesizations which allows us to keep track of some additional features of these combinatorial objects.

\begin{definition}\label{lem:matching algorithm}
    Let $(F,L)$ be a balanced spaced parenthesization of $n$ spaces with $|F|=|L|=k$.
    We assign a labeling to the $k$ pairs of matching parenthesis, which we denote by $\{(f_i,\ell_i)\}_{i\in[k]}$, and the values are determined as follows. For for each $i\in[k]$, consider the $i$-th pair of matching parenthesis:
    \begin{itemize}
        \item The value $f_i$ is equal to the number of the space immediately after the $i$-th opening parenthesis, and 
        \item the value $\ell_i$ is equal to the number of the space immediately before the $i$-th closing parenthesis (using the standard Catalan parenthesization \cite{catalan_ex}). 
    \end{itemize}
\end{definition}

From \Cref{ex:balanced ex}, we had the following balanced spaced parenthesization:

\[\begin{smallmatrix}(\_&\_&(\_&\_&(\_)&\_)&\_)\\\,1&2&\,\,3&4&\,5\,&6\,\,&7\,\end{smallmatrix}.\]
Using \Cref{lem:matching algorithm}, we have the following $(f_i, \ell_i)$ pairs: $(1,7)$, $(3,6)$ and $(5,5)$.

\begin{lemma}\label{lem:matching_pairs_new}
    \Cref{lem:matching algorithm} uniquely determines $(F,L)$, and the labeled pairs  satisfy:
    \begin{itemize}
        \item $F=\{f_i\}_{i\in[k]}$ and $L=\{\ell_i\}_{i\in[k]}$
        \item $f_i\le \ell_i$ (pairs open and then close) for all $i\in[k]$, and
        \item we avoid the pattern $f_i<f_j\le \ell_i<\ell_j$ (i.e., the matching pairs nest).
    \end{itemize}
    If needed, we may assume that $f_i<f_j$ whenever $i<j$ to provide a canonical indexing of the matching pairs. Under this assumption, these properties give an equivalent definition of a balanced spaced parenthesization.
\end{lemma}

\begin{proof}
    The three points above are equivalent to the standard Catalan parenthesization definition.
\end{proof}

\begin{figure}[H]
    \centering
    \includegraphics[width=0.75\linewidth]{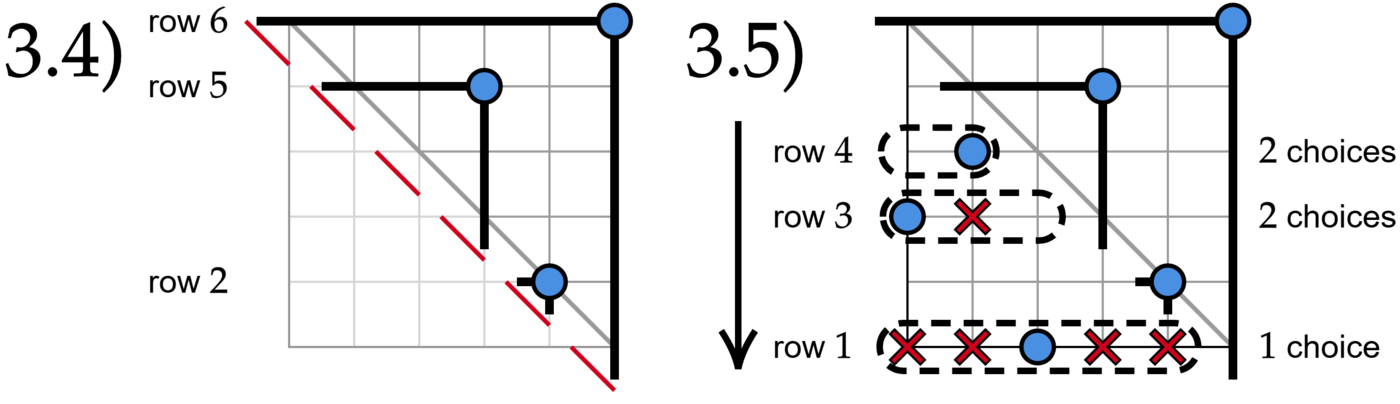}
    \caption{In \Cref{lem:balanced -> non-intersecting partial} (left), we place peaks in rows $i$ satisfying $n-i+1\in F$. In \Cref{lem:Phi is surjective} (right), we place the remaining entries below the antidiagonal. Here, there are four possible outcomes of Lehmer parking functions with the given parenthesization.}
    \label{fig:algorithm for below antidiagonal}
\end{figure}

We now convert matching pairs of labeled parentheses into entries in a partial arm-leg diagram, this in turn establishes \Cref{lem:balanced -> non-intersecting partial}.

\begin{proof}[Proof of \Cref{lem:balanced -> non-intersecting partial}]
    From $(F,L)$ we obtain the matching pairs $\{(f_i,\ell_i)\}_{i\in[k]}$ as in \Cref{lem:matching algorithm}.
    For each pair $(f_i,\ell_i)$, we then add the entry $(\ell_i,n-f_i+1)$ to a partial arm-leg diagram $\tau$. 
    Henceforth, we set $\tau_{\ell_i}=n-f_i+1$.
    We now verify that $\tau$ is the unique non-intersecting partial arm-leg diagram such that $(\arms_\tau,\legs_\tau)=(F,L)$:
    \begin{itemize}
        \item Each entry lies at or above the antidiagonal because $f_i\le \ell_i$ implies $\tau_{\ell_i}=n-f_i+1\ge n-\ell_i+1$.
        \item We have $\legs_\tau=\{i\,:\,(i,\tau_i)\text{ is at or above the antidiagonal}\}=\{\ell_i\}_{i\in[k]}=L$. Similarly, $\arms_\tau=F$.
        \item No row or column contains more than one entry since $F=\{f_i\}_{i\in[k]}$ and $L=\{\ell_i\}_{i\in[k]}$ have $k$ distinct elements.
        \item $\tau$ is non-intersecting because otherwise, by \Cref{thm: characterization of inversion table outcomes}, $\tau$ contains the pattern
        \[n-\ell_j+1<n-\ell_i+1\le \tau_{\ell_j}<\tau_{\ell_i}\] for some $\ell_i,\ell_j\in[n]$. 
        Replacing $\tau_{\ell_i}=n-f_i+1$ and $\tau_{\ell_j}=n-f_j+1$ in the set of inequalities yields:
        \[n-\ell_j+1<n-\ell_i+1\le n-f_j+1<n-f_i+1,\] which implies
        $f_i<f_j\le \ell_i<\ell_j$
        for some $f_i,f_j$. But this pattern is avoided by \Cref{lem:matching_pairs_new}.
    \end{itemize}
    This partial arm-leg diagram is unique.
\end{proof}

Thus, a balanced spaced parenthesization $(F,L)$ can be assigned a unique set of non-intersecting peaks of a partial arm-leg diagram $\tau$ with $(\arms_\tau,\legs_\tau)=(F,L)$. 
Next we show in \Cref{lem:Phi is surjective} that there exists an outcome of a Lehmer parking function, $\pi$, with peaks given by $\tau$, finishing the proof that the map $\Phi$ defined in \Cref{eq:Phi} is surjective. We emphasize that such $\pi$ only exists because $\tau$ comes from a balanced spaced parenthesization; if $\tau$ is the empty partial arm-leg diagram, for example, then $\tau$ is not the set of peaks of an arm-leg diagram of any permutation.

\begin{lemma}\label{lem:Phi is surjective}
    If $(F,L)$ is a balanced spaced parenthesization, then there exists at least one 
    $\pi \in \Out(\IT_n)$ such that $(F,L)=(\arms_\pi,\legs_\pi)$.
\end{lemma}
\begin{proof}
    Given $(F,L)$, we start from the non-intersecting partial arm-leg diagram $\tau$ from \Cref{lem:balanced -> non-intersecting partial} satisfying $(\arms_\tau,\legs_\tau)=(F,L)$ (see \Cref{fig:algorithm for below antidiagonal}, left). Because $\arms_\tau=F$, we have that $\tau$ contains entries in each row $n-i+1$ such that $i\in F$.
    We now construct $\pi$ from $\tau$ as follows: we focus on the rows which do not contain an entry, and in each such row 
    beginning from top to bottom, iteratively place an entry $(j,n-i+1)$ below the antidiagonal, where $j\in[1,i-1]$ and $i\notin F$,
    (see \Cref{fig:algorithm for below antidiagonal}, right).
    We show that if $n-i+1$ is the uppermost empty row, 
    there exists at least one column numbered $j\in[1,i-1]$ which contains no entries of the diagram. This implies that we can place the point $(j,n-i+1)$ below the antidiagonal.
    Note that \Cref{lem:d_i = leg intersections} follows identically for $(\arms_\tau,\legs_\tau)$; that is, $d_i$ is given by the number of entries of $\tau$ in $[i,n]\times[n-i+1,n]$.
    We  illustrate the following equations in \Cref{fig:proof_3.6}, where the figure numbers correspond to the equation numbers below. In the caption we provide some additional details to explain the following count: 
    The number of empty columns $j\in[1,i-1]$ is given by 
    \begin{align}
        \label{eq:2}&(i-1)-\text{number of entries in }[1,i-1]\times[n-i+2,n]&\text{no lower entries below the antidiagonal}\\
        \label{eq:3}=\,&\text{number of entries in }[i,n]\times[n-i+2,n]&\text{the top $i-1$ rows have $i-1$ entries}\\
        \label{eq:4}=\,&\text{number of entries in }[i,n]\times[n-i+1,n]&\text{row $n-i+1$ is empty}\\
        =\,&d_i,&\text{by \Cref{lem:d_i = leg intersections} for partial diagrams}
    \end{align}
    which is positive because $(\arms_\tau,\legs_\tau)$ is balanced. 
    Placing entries $(j,n-i+1)$ in this manner for each remaining row yields an outcome of a Lehmer parking function, $\pi$, satisfying $(\arms_\pi,\legs_\pi)=(F,L)$.
\end{proof}

\begin{figure}[H]
    \centering
    \includegraphics[width=.9\linewidth]{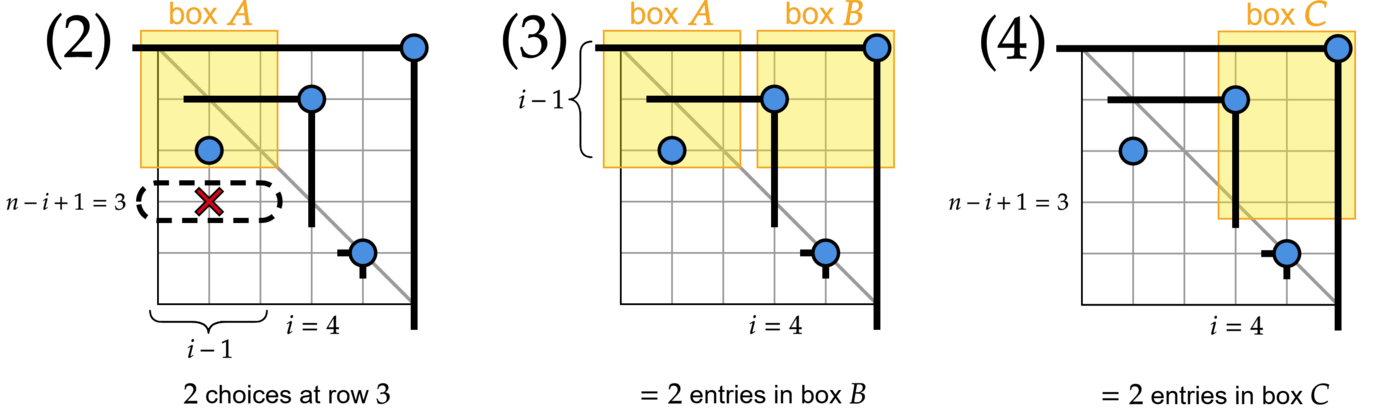}
    \caption{We assume that row $n-i+1$ is the uppermost empty row. There are $i-1$ entries to the left of the antidiagonal entry $(i,n-i+1)$, and $(i-1)-|\text{box }A|$ are available for the next entry (\cref{eq:2}). Likewise, there are $i-1$ entries in the top $i-1$ columns, and $(i-1)-|\text{box }A|$ of them lie in box $B$ (\cref{eq:3}). Finally, row $n-i+1$ is empty, so the same number of entries lie in box $C$ (\cref{eq:4}), the box given in \Cref{lem:d_i = leg intersections}.}
    \label{fig:proof_3.6}
\end{figure}

In terms of our function $\Phi$ from \Cref{eq:Phi}, \Cref{lem:Phi is surjective} tells us that $\Phi$ is a well-defined surjective function. 
While \Cref{lem:Phi is surjective} proves existence of an outcome with the given properties, the same construction given in the proof of \Cref{lem:Phi is surjective} actually gives a count for all such outcomes.
We establish this next. 

\begin{corollary}
    If $(F,L)$ is a balanced spaced parenthesization, then the cardinality of the fiber is
    \[|\Phi^{-1}((F,L))|=\displaystyle\prod_{n-i+1\notin F}d_i.\]
\end{corollary}
\begin{proof}
    In the proof of \Cref{lem:Phi is surjective}, the number $d_i$ of choices for the entry $(j,n-i+1)$ in each row $n-i+1$ such that $n-i+1\notin F$ is independent of the previous choices, and we can choose any of the $d_i$ available columns at each step. In particular, we can parameterize the exact set of $\pi\in\Phi^{-1}((F,L))$ by specifying an integer $g_i\in[1,d_i]$ for each row $n-i+1$ such that $i\notin F$ directing the $g_i$-th smallest choice to be made for $(j,n-i+1)$.
\end{proof}

For example, in \Cref{fig:algorithm for below antidiagonal}, we make the second smallest choice for the third row from the top, so here $g_3=2$. Similarly, we have $g_4=1$ and $g_6=1$. We can display these digits in the parenthesization in spaces without opening parentheses: here $(\_\ (\_\ \underline{2}\ \underline{1})(\_)\ \underline{1})$ corresponds to $\pi=314526$. 
This allows us to introduce the following.

\begin{definition}\label{def:gBSP}
    A $g$-\textit{balanced spaced parenthesization} $(F,L,g)$ of length $n$ consists of two subsets $F,L\subseteq [n]$ of equal size, $F$ indexing opening parentheses and $L$ indexing closing parentheses, and a tuple $g=(g_i)_{n-i+1\notin F}$ such that $d_i=|F\cap[1,i]|-|L\cap[1,i-1]|$ is positive and $g_i\in[1,d_i]$ for all $i\in[n]$ such that $n-i+1\notin F$. We let $\GBSP_n$ denote the set of $g$-balanced spaced parenthesizations of length $n$, and we let $\Phi':\Out(\IT_n)\to\GBSP_n$ denote the function taking an outcome of a Lehmer parking function to the particular $g$-balanced spaced parenthesization corresponding to it.
\end{definition}

Next, we show that the map $\Phi'$ is in fact a bijection.

\begin{lemma}\label{lem:ITO <-> BSP}
    The set of all outcomes of Lehmer parking functions of length $n$ are in bijection with $g$-balanced spaced parenthesizations of $n$ spaces.
\end{lemma}
\begin{proof}
In \Cref{lem:Phi is surjective}, we established that $\Phi$ is surjective, and $\Phi'$ inherits this property. Therefore, we need to establish that $\Phi'$ is also injective.
Suppose that $\Phi'(\sigma)=\Phi'(\sigma')$, with $\sigma\neq \sigma'$.
Then, let
$i$ be the smallest index at which $\sigma_i\neq \sigma_i'$.
If $(i,\sigma_i)$ is above the antidiagonal, then the two permutations must have different arm-leg sets.
This implies $(\arms_\sigma,\legs_\sigma)$ and $(\arms_{\sigma'},\legs_{\sigma'})$ are distinct, contradicting that $\Phi'(\sigma)=\Phi'(\sigma')$.
In the case that $(i,\sigma_i)$ is below the antidiagonal, then the element $g_{n-\sigma_i+1}$ in $g$, for $\sigma$, differs from  $g_{n-\sigma_i+1}'$ in $g'$, for $\sigma'$, again contradicting that 
$\Phi'(\sigma)=\Phi'(\sigma')$.
Therefore, $\Phi'$ is injective.
This shows that $\Phi'$ is a bijection.
\end{proof}

\subsection{Set Partitions and Balanced Spaced Parenthesizations}\label{sec:setpartition}
A set partition of $[n]$ is a set of nonempty disjoint subsets, or \textit{blocks}, whose union is $[n]$. For example, $\{\{1,4\},\{5\},\{2,3,6\}\}$ is a set partition of $[6]$ with $|\mB|=3$ blocks. 
Given a set partition $\mB$, we define $\minB$ as the set of minima of the blocks of $\mB$ and $\maxB$ as the set of maxima of the blocks of $\mB$. 
For the above set partition $\mB$, we have $\minB=\{1,2,5\}$ and $\maxB=\{4,5,6\}$. 
We let $\Pi_n$ denote the set of set partitions of $[n]$.

\begin{lemma}\label{lem:set partitions are balanced}
    If $\mB$ is a set partition of $[n]$, then the spaced parenthesization $(\minB,\maxB)$ is balanced.
\end{lemma}
\begin{proof}
    For $i\in[n]$ we show that $d_i=|\minB\cap[1,i]|-|\maxB\cap[1,i-1]|$ is positive. It is nonnegative because every block with maximum in $\maxB\cap[1,i-1]$ has minimum in $\minB\cap[1,i]$. It is positive because the block containing $i$ has minimum in $\minB\cap[1,i]$ and maximum not in $\maxB\cap[1,i-1]$.
\end{proof}

Consider the map 
\begin{align}
\Psi:\Pi_n\to\BSP_n, &\qquad \mbox{defined by }\mB\mapsto(\minB,\maxB).
\label{def:Psi}
\end{align}
\Cref{lem:set partitions are balanced} establishes that $\Psi$ is well-defined. 
To establish the surjectivity of $\Psi$ we recall \Cref{lem:matching algorithm} for the balanced spaced parenthesizations. 

\begin{lemma}\label{lem:BSP to SP}
    If $(F,L)$ is a balanced spaced parenthesization in $\BSP_n$, then there exists a set partition $\mB\in\Pi_n$ such that $(F,L)=(\minB,\maxB)$.
\end{lemma}
\begin{proof}
This proof is constructive.
We begin with a balanced spaced parenthesization, built as pairs $(f_i, \ell_i)$, as in \Cref{lem:matching algorithm}.
From each pair $(f_i,\ell_i)$ we construct a set $B_i=\{f_i,\ell_i\}$, where $f_i=\min B_i$ and $\ell_i=\max B_i$.
Note that we have now constructed $|F|=|L|=k$ blocks $B_1,B_2,\ldots, B_{k}$ each of size $1$ or $2$.

In the case that $F\cup L=[n]$, we are done. 
On the other hand, 
we need to include all remaining numbers into the blocks $B_1,B_2,\ldots,B_k$ iteratively, looking at the elements of the set $I=[n]\setminus(F\cup L)$ from smallest to largest. 

At some point in our process, we want to place $i\in I$ in one of the existing $k$ blocks. We first claim that there will always be at least one block $B_y$ such that $j\in [f_y+1, \ell_y-1]$.
From the balanced spaced parenthesization structure, we will have $d_i=|F\cap[1,i]|-|L\cap[1,i-1]|$ options. At spot $i$ (in the balanced spaced parenthesization), the number $d_i$ tells us specifically how many parenthesis pairs spot $i$ is contained in. 
Since these determine the minimums and maximums of the $k$ blocks, there are $d_i$ blocks with minimums less than $i$ and maximums greater than $i$.
\end{proof}

\begin{remark}
    Our process in \Cref{lem:BSP to SP} means that, in contrast to \Cref{sec:invPFoutcome}, we are not encountering a choice related to placing some of the elements $i$ such that $d_i=1$. Consider the example in the last section of $(\_\ (\_\ \underline{2}\ \underline{1})(\_)\ \underline{1})$ which can correspond to the set partition $\{\{1,4\},\{2,3,6\},\{5\}\}$. In this case, 
    the $d_3=2$ tells us we have two choices for placing the value $3$ in the set partition we are constructing, while we have already made the single choice available for placing the $4$ and $6$ in the set partition.
\end{remark}

The set partition construction presented in the proof of \Cref{lem:BSP to SP} along with the remark following it, help us arrive at the next result.

\begin{corollary}
    If $(F,L)$ is a balanced spaced parenthesization, then the cardinality of the fiber is 
    \[
    |\Psi^{-1}((F,L))| = \prod_{i\notin F}d_i.
    \]

\end{corollary}
\begin{proof}
    In the proof of \Cref{lem:BSP to SP}, the number $d_i$ of choices of blocks for each $i\notin F$ is independent of the previous choices, and we can choose any of the $d_i$ valid blocks at each step. In particular, we can parameterize the exact set $\Psi^{-1}((F,L))$ of such $\mB$ by specifying an integer $g_i\in[1,d_i]$ for each $i\notin F$ directing $i$ to be placed in the valid block with the $g_i$-th smallest minimum.
\end{proof}

We are now ready to show that $\Psi'$ is also a bijection.

\begin{lemma}\label{lem:Pi <-> BSP}
    Set partitions of $[n]$ are in bijection with $g$-balanced spaced parenthesizations of $n$ spaces.
\end{lemma}

\begin{proof}
In \Cref{lem:BSP to SP}, we established that $\Psi$ is surjective, and $\Psi'$ inherits this property. 
Therefore, we need to establish that $\Psi'$ is also injective.
Suppose that $\Psi'(\mathcal{B})=\Psi'(\mathcal{B}')$, with $\mathcal{B}\neq \mathcal{B}'$.
Thus, let
$i$ be the smallest element such that the subset containing $i$ in $\mathcal{B}$, denoted $B_j$,  is different from the subset containing $i$ in $ \mathcal{B}'$, denoted  $B_k'$.

One of the following must occur: either $B_j$ and $B_k'$ have different minimal or maximal elements, or some of the internal elements of the subsets are different. 

If the minimal or maximal elements differ, then the $(F,L)$ set for $\mathcal{B}$ is different from the $(F',L')$ set for $\mathcal{B}'$. This would make the balanced spaced parenthesizations distinct, contradicting the asusmption that $\Psi'(\mB)=\Psi'(\mB')$. 
If the internal elements of the subsets differ, then the element $g_i$ in $g$, for $\mathcal{B}$,  differs from  $g_i'$ in $g'$, for $\mathcal{B}'$. 
This also contradicts that $\Psi'(\mathcal{B})=\Psi'(\mathcal{B}')$. 
Therefore, $\Psi'$ is injective.
This shows that $\Psi'$ is a bijection.
\end{proof}

The composition of the bijections from \Cref{lem:ITO <-> BSP} and \Cref{lem:Pi <-> BSP} prove the following:

\begin{proof}[Proof of \Cref{thm:main bijection}]
    The set of outcomes of Lehmer parking functions of length $n$ and the set partitions of $[n]$ are in bijection.
\end{proof}

\section{Lehmer Parking Functions and Catalan Numbers}\label{sec:catalan}

In this section, we discuss two areas related to outcomes of Lehmer parking functions where we encounter the Catalan numbers \cite[\seqnum{A000108}]{OEIS}.

We consider the set of weakly decreasing Lehmer parking functions, which we denote by
\[\IT_n^\downarrow=\{\alpha=(a_1,a_2,\ldots,a_n)\in\IT_n:~a_i\geq a_{i+1}\mbox{ for all }i\in[n-1]\}.\] 
By \cite[Lemma 10]{akc}, the set of permutations  in $\mathfrak{S}_n$ which avoid $132$ is in bijection with the set of weakly decreasing inversion tables of length $n$. 
Recall that 
given $\sigma \in \mathfrak{S}_n$ and $\rho \in S_m$ we say that $\sigma$ avoids $\rho$ as a pattern if there does not exist $0 \leq i_0 < i_1 < \cdots <i_m \leq (n-1)$ such that $\sigma_{i_a} \leq \sigma_{i_b}$ if and only if $\rho_a \leq \rho_b$. 
Let $\Av_n(132)$ be the set of permutations in $\mathfrak{S}_n$ that avoid $132$.
It is known that there are Catalan many $132$-avoiding permutations in $\mathfrak{S}_n$, see \cite{stanley_2015}. 
Thus $|\IT_n^\downarrow|=C_n$.

Now let $\PF_n^\downarrow$ be the set of weakly decreasing parking functions of length $n$. 
It is known that the set $|\PF_n^\downarrow|=C_n$, see \cite{stanley_2015}.
Of course, $\IT_n^\downarrow\subseteq \PF_n^\downarrow$ since Lehmer parking functions are parking functions. Moreover, as the sets have the same cardinality, we have that $\IT_n^\downarrow=\PF_n^\downarrow$.

In the following two technical results we show that the outcomes of weakly decreasing inversion tables avoid the pattern $132$ and that the outcome map from $\IT_n^\downarrow$ to $\Av_n(132)$ is an injective map.

\begin{proposition}\label{thm:cat}
    Let $n\geq 1$. Then $\Out(\IT_n^\downarrow)\subseteq\Av_n(132)$.
\end{proposition}
\begin{proof}
    Suppose for contradiction that $\pi$ is the outcome of a weakly decreasing Lehmer parking function and $\pi=\pi_1\pi_2\cdots\pi_n$ contains the pattern $132$. 
    This means, that there exist $\pi_i<\pi_j>\pi_k$ with $\pi_i<\pi_k$ and $i<k<j$.
This implies that after cars $\pi_i$ and $\pi_k$ park, in order to contain the $132$ pattern, there is a gap on the street between spots $i$ and $k$. 
However, the only way for such a gap to exist is if the preference for car $\pi_k$ is strictly larger than the preference for car $\pi_{k-1}$. However, this contradicts the fact that the preferences are weakly decreasing, giving a contradiction. Thus the outcome of a weakly increasing parking function avoids the pattern $132$. 
\end{proof}

\begin{lemma}\label{lem:injective}
The outcome map $\Out: \IT_n^\downarrow\to \Av_n(132)$, defined by parking the cars using a weakly decreasing Lehmer parking function, is an injective map. 
\end{lemma}
\begin{proof}
By \Cref{thm:cat}, we know that the outcome of a weakly decreasing Lehmer parking function is a 132 avoiding permutation.
Suppose there exist $\alpha$ and $\beta$ both in $\IT_n^\uparrow$ such that $\Out(\alpha)=\Out(\beta)=\pi$ but $\alpha\neq\beta$.
Let $i\in[n]$ be the smallest index such that $a_i\neq b_i$.
Without loss of generality assume that $a_i<b_i$. 
Since $a_i<b_i$ and both $\alpha$ and $\beta$ are weakly decreasing, $a_i$ is now smaller than all of the preferences $a_1,a_2,\ldots,a_{i-1}$,  so car $i$ parks to the left of cars $1,2,\ldots,i-1$. In addition, since $a_i<b_i$ then car $i$ under $\alpha$ parks to the left of where car $i$ parks under $\beta$.
Then this implies that the outcome for car $i$ is not the same under each of the preferences $\alpha$ and $\beta$, contracting that $\Out(\alpha)=\Out(\beta)$.
\end{proof}

Since both  sets $\IT_n^\downarrow$ and $\Av_n(132)$ have cardinality $C_n$,  \Cref{lem:injective} implies the following.

\begin{theorem}\label{thm:weakly_dec_lehmer}
If $n\geq 1$, then $\Out(\IT_n^\downarrow)=\Av_n(132)$.    
\end{theorem}

\bibliographystyle{plain}
\bibliography{bibliography}

\end{document}